\documentclass[a4paper,12pt]{article}

\usepackage[cp1250]{inputenc}
\usepackage{fancyhdr}
\usepackage{amssymb,amsmath,amsfonts,theorem,dsfont,a4wide}
\usepackage[T1]{fontenc}
\usepackage{epsfig}

\usepackage{color}
\usepackage[normalem]{ulem}

\usepackage{float}
\usepackage{lineno}

\newtheorem{lemma}{Lemma}[section]
\newtheorem{theorem}{Theorem}[section]

\newenvironment{proof}
   {\begin{trivlist}\item[]\textbf{\bf{Proof. }}\ignorespaces}
   {\qed\end{trivlist}}

\newcommand{\qed}{{\ifmmode q.e.d. \else\unskip\nobreak\hfil
\penalty50\quad\null\nobreak\hfill $\square$ \parfillskip=0pt
\finalhyphendemerits=0\par\fi}}

\begin{document}
\begin{center}
{\bf \Large Three classes of 1-planar graphs}\\[3mm]
{\bf J\'ulius Czap and Peter \v Sugerek}\\
Department of Applied Mathematics and Business Informatics, Faculty of Economics\\ Technical University of Ko\v{s}ice, N\v{e}mcovej 32, 040 01 Ko\v{s}ice, Slovakia\\
email: julius.czap@tuke.sk, peter.sugerek@tuke.sk\\
\end{center}

\bigskip
{\bf Abstract:} A graph is called 1-planar if it can be drawn in the plane so that each of its edges is crossed by at most one other edge. In this paper we decompose the set of all 1-planar graphs into three classes $\mathcal C_0, \mathcal C_1$ and $\mathcal C_2$ with respect to the types of crossings and present the decomposition of 1-planar join products. 

Zhang \cite{z} proved that every $n$-vertex 1-planar graph of class $\mathcal C_1$ has at most $\frac{18}{5}n-\frac{36}{5}$ edges and a $\mathcal C_1$-drawing with at most $\frac 35 n-\frac 65$ crossings. We improve these results. We show that every $\mathcal C_1$-drawing of a 1-planar graph has at most $\frac 35 n-\frac 65$ crossings. Consequently, every $n$-vertex 1-planar graph of class $\mathcal C_1$ has at most $\frac{18}{5}n-\frac{36}{5}$ edges. Moreover, we prove that this bound is sharp.



\section{Introduction}
All graphs considered in this paper are finite, simple and undirected, unless otherwise stated. We use $V(G)$ and $E(G)$ to denote the vertex set and the edge set of a graph $G$, respectively. The \emph{crossing number} of $G$, denoted by $cr(G)$, is the minimum possible number of crossings in a drawing of $G$ in the plane.

A drawing of a graph is \emph{1-planar} if each of its edges is crossed at most once. If a graph has a 1-planar drawing, then it is \emph{1-planar}. Let $G$ be a 1-planar graph drawn in the plane so that no of its edges is crossed more than once. The \emph{associated plane graph} $G^\times$ of $G$ is the plane graph obtained from $G$ such that the crossings of $G$ become new vertices of degree four; we call these vertices \emph{false}. Vertices of $G^\times$ which are also vertices of $G$ are called \emph{true}. Similarly, the edges and faces of $G^\times$ are called false, if they are incident with a false vertex, and true otherwise. For a false vertex $c$ let $N_{G^\times}(c)$ denote the set of neighbors of $c$ in $G^\times$. 

It is easy to see that if a graph has a 1-planar drawing in which two edges $e_1, e_2$ with a common endvertex cross, then the drawing of $e_1$ and $e_2$ can be changed so that these two edges no longer cross. Therefore, we may assume that adjacent edges never cross and that no edge is crossing itself. Consequently, every crossing involves two edges with four distinct endvertices, i.e.  $|N_{G^\times}(c)|=4$ for every false vertex $c$.


We say that a 1-planar graph is of class $\mathcal C_0$ if it has such a 1-planar drawing $D$ that for any two false vertices $c_1, c_2$ of $D^\times$ it holds $|N_{D^\times}(c_1)\cap N_{D^\times}(c_2)|=0$. This class of 1-planar graphs was investigated in \cite{ks, zl, zly} under the notion plane graphs with independent crossings.
We say that a 1-planar graph is of class $\mathcal C_i$, $i\in\{1,2\}$, if it is not of class $\mathcal C_k$ for any $k<i$ and it has such a 1-planar drawing $D$ that for any two false vertices $c_1, c_2$ of $D^\times$ it holds $|N_{D^\times}(c_1)\cap N_{D^\times}(c_2)|\le i$.
The corresponding drawing is called \emph{$\mathcal C_i$-drawing}, $i=0,1,2$. 

Note that, the class $\mathcal C_1$ was investigated in \cite{z} under the notion plane graphs with near-independent crossings. 

In this paper we show that every 1-planar graph belongs to one of the classes $\mathcal C_0$, $\mathcal C_1$ and $\mathcal C_2$. After that we deal with the classification of 1-planar joins. The \emph{join product} (or shortly, join) $G+H$ of two graphs $G$ and $H$ is obtained from vertex–disjoint copies of $G$ and $H$ by adding all edges between $V(G)$ and $V(H)$.

The author of \cite{z} proved that any \emph{good $\mathcal C_1$-drawing} (that is, a $\mathcal C_1$-drawing with minimum possible number of crossings) of an $n$-vertex 1-planar graph of class $\mathcal C_1$ has at most $\frac 35 n-\frac 65$ crossings. In this paper we improve this result. We show that this bound holds for any $C_1$-drawing. From this result it follows that any $n$-vertex 1-planar graph of class $\mathcal C_1$ has at most $\frac{18}{5}n-\frac{36}{5}$ edges. We show that this bound is tight.

The disjoint union of two graphs $G_1$ and $G_2$ will be denoted by $G_1 \cup G_2$ and the disjoint union of $k$ isomorphic graphs $G_1$ will be denoted by $kG_1$.

\section{Results}

First we show that every 1-planar graph $G$ has such a 1-planar drawing $D$ that for any two false vertices $c_1, c_2$ of $D^\times$ it holds $|N_{D^\times}(c_1)\cap N_{D^\times}(c_2)|\le 2$.


Assume that there are crossings $c_1,c_2$ in a 1-planar drawing $D$ such that for the corresponding false vertices it holds $|N_{D^\times}(c_1)\cap N_{D^\times}(c_2)|\ge3$.  Let $xy$ and $zw$ be the edges which cross at $c_1$. Since $|N_{D^\times}(c_1)\cap N_{D^\times}(c_2)|\ge3$, without loss of generality, we can assume that the crossing $c_2$ is the interior point of the edge $xz$. In this case we can redraw the edge $xz$ such that it is crossing-free by following the edges that cross at $c_1$ from $x$ and $z$ until they meet in a close neighborhood of $c_1$. Therefore, if $D^\times$ contains such false vertices $c_1,c_2$ that $|N_{D^\times}(c_1)\cap N_{D^\times}(c_2)|\ge3$, then we can eliminate one of them.

\medskip
In the following we deal with the classification of 1-planar joins.  

\begin{lemma}\label{hmco}
Let $W$ be a 1-planar graph of class $\mathcal C_0$. Then any $\mathcal C_0$-drawing of $W$ contains at most $\frac{|V(W)|}{4}$ crossings.
\end{lemma}

\begin{proof}
It follows from the definition of $\mathcal C_0$-drawing.
\end{proof}

\begin{lemma}\label{crossing}
Let $W$ be a 1-planar graph of class $\mathcal C_1$. If $W$ has at most 8 vertices, then any $\mathcal C_1$-drawing of $W$ has at most two crossings.
\end{lemma}

\begin{proof}
Let $c_1,c_2,c_3$ be crossings in a $\mathcal C_1$-drawing $D$ of $W$. Clearly, $|N_{D^\times}(c_1) \cup N_{D^\times}(c_2)|\ge 7$, since $D$ is a $\mathcal C_1$-drawing. Therefore, there is at most one true vertex in $D^\times$ which is not incident neither $c_1$ nor $c_2$.  The false vertex $c_3$ is incident with at most one vertex in $N_{D^\times}(c_1)$ and with at most one vertex in $N_{D^\times}(c_2)$. Consequently, $c_3$ has at most three (true) neighbors, a contradiction.
\end{proof}

\begin{theorem}\cite{kleitman}\label{kleitman}
Let $K_{m,n}$ denote the complete bipartite graph on $m+n$ vertices. Then
$cr(K_{m,n})=\left\lfloor \frac m2 \right\rfloor\left\lfloor\frac{m-1}{2}\right\rfloor \left\lfloor \frac n2\right\rfloor\left\lfloor\frac{n-1}{2}\right\rfloor$
for $\min\{m, n\}\le6$.
\end{theorem}

\begin{lemma}\label{44}
If $|V(G)|\ge|V(H)|\ge 4$ and $G+H$ is 1-planar, then $G+H$ is of class $\mathcal C_2$.
\end{lemma}

\begin{proof}
If $|V(G)|\ge|V(H)|\ge 4$, then $G+H$ contains $K_{4,4}$ as a subgraph. The graph $K_{4,4}$ is 1-planar, see \cite{ch}. From Theorem \ref{kleitman} we have $cr(K_{4,4})=4$, therefore any 1-planar drawing of $K_{4,4}$ contains at least four crossings. Hence, Lemmas \ref{hmco} and \ref{crossing} imply that the graph  $K_{4,4}$ is of class $\mathcal C_2$ . The fact that $G+H$ contains a subgraph of class $\mathcal C_2$ implies that $G+H$ also belongs to $\mathcal C_2$.
\end{proof}

\begin{lemma}\label{53}
If $|V(G)|\ge 5$, $|V(H)|\ge 3$ and $G+H$ is 1-planar, then $G+H$ is of class $\mathcal C_2$.
\end{lemma}

\begin{proof}
In this case $G+H$ contains $K_{5,3}$ as a subgraph. The graph $K_{5,3}$ is 1-planar, see~\cite{ch}.
The crossing number of $K_{5,3}$ is four (see Theorem \ref{kleitman}), hence (by Lemmas \ref{hmco} and \ref{crossing}) it is of class $\mathcal C_2$.
Consequently, the supergraph $G+H$ of $K_{5,3}$ is also of class $\mathcal C_2$.
\end{proof}

From Lemmas \ref{44} and \ref{53} we obtain, that there are only three possibly cases for joins of classes $\mathcal C_0$ and $\mathcal C_1$, namely:
\begin{itemize}
\item $|V(G)|=|V(H)|=3$.
\item $|V(G)|=4$ and $|V(H)|=3$.
\item $|V(G)|\ge5$ and $|V(H)|\le 2$.
\end{itemize}

\subsection{The first case: $|V(G)|=|V(H)|=3$}

If the graphs $G$ and $H$ have together (at most) six vertices, then $G+H$ is always 1-planar, since it is a subgraph of the complete graph on six vertices $K_6$ which is 1-planar, see \cite{ch}. 

\begin{lemma}\label{c0c2}
If $W$ is a 1-planar graph on at most six vertices, then $W$ is either of class $\mathcal C_0$ or $\mathcal C_2$.
\end{lemma}

\begin{proof}
If $W$ has a 1-planar drawing with at most one crossing, then it is of class $\mathcal C_0$. If any 1-planar drawing $D$ of $W$ has at least two crossings, say $c_1,c_2$, then it is of class $\mathcal C_2$, since $|N_{D^\times}(c_1)\cap N_{D^\times}(c_2)|>1$. 
\end{proof}

Let $C_n$ and $P_n$ denote the cycle and the path on $n$ vertices, respectively.

\begin{lemma}
The graphs $C_3+P_2\cup P_1$ and $P_3+P_3$ are of class $\mathcal C_0$.
\end{lemma}

\begin{proof}
$\mathcal C_0$-drawings of the graphs $C_3+P_2\cup P_1$ and $P_3+P_3$ are shown in Figure \ref{figurey}.
\end{proof}

\begin{figure}
\centerline{
\begin{tabular}{ccc}
\includegraphics{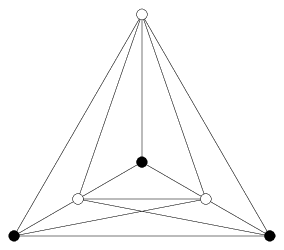}&&
\includegraphics{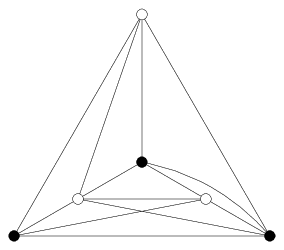}
\\$C_3+P_2\cup P_1$&&$P_3+P_3$
\end{tabular}
}
\caption{$\mathcal C_0$-drawings of the graphs $C_3+P_2\cup P_1$ and $P_3+P_3$.}
\label{figurey}
\end{figure}

%
%
%

The crossing number of join products of cycles and paths were studied in \cite{klesc}.

\begin{theorem}\cite{klesc}\label{klescpc}
$cr(C_n+P_m)=\left\lfloor \frac m2 \right\rfloor\left\lfloor\frac{m-1}{2}\right\rfloor \left\lfloor \frac n2\right\rfloor\left\lfloor\frac{n-1}{2}\right\rfloor+1$ for $m\ge2,n\ge3$ with $\min\{m, n\}\le6$.
\end{theorem}

\begin{lemma}\label{c3p3}
The graph $C_3+P_3$ is of class $\mathcal C_2$.
\end{lemma}

\begin{proof}
The join $C_3+P_3$ is 1-planar, since it is a subgraph of $K_6$, which is 1-planar. From Theorem \ref{klescpc} it follows $cr(C_3+P_3)=2$. Hence, Lemmas \ref{hmco} and \ref{c0c2} imply that $C_3+P_3$ is of class $\mathcal C_2$.
\end{proof}

\subsection{The second case: $|V(G)|=4$ and $|V(H)|=3$}

%

\begin{lemma}\label{c1c2}
If $|V(G)|=4$ and $|V(H)|=3$, then the graph $G+H$ cannot be of class $\mathcal C_0$.
\end{lemma}

\begin{proof}
The graph $G+H$ contains $K_{4,3}$ as a subgraph whose crossing number is two (by Theorem~\ref{kleitman}). This means that any drawing of $G+H$ contains at least two crossings. Therefore, $G+H$ cannot be of class $\mathcal C_0$ (see Lemma \ref{hmco}).
\end{proof}

\begin{lemma}
The graphs $P_4+P_3$ and $2P_2+C_3$ are of class $\mathcal C_1$.
\end{lemma}

\begin{proof}
From Lemma \ref{c1c2} it follows that these graphs cannot be of class $\mathcal C_0$. $\mathcal C_1$-drawings of the graphs $P_4+P_3$ and $2P_2+C_3$ are shown in Figure \ref{figurex}.
\end{proof}

\begin{figure}
\centerline{
\begin{tabular}{ccc}
\includegraphics[angle=90]{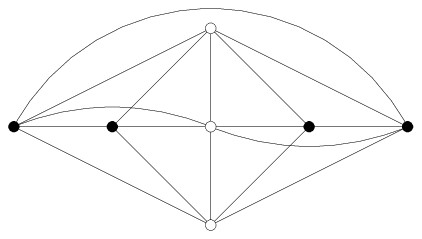}&&
\includegraphics{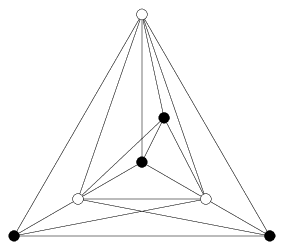}
\\$P_4+P_3$&&$2P_2+C_3$
\end{tabular}
}
\caption{$\mathcal C_1$-drawings of the graphs $P_4+P_3$ and $2P_2+C_3$.}
\label{figurex}
\end{figure}

\begin{lemma}\label{3vertex}
Let  $|V(G)|=4$ and $|V(H)|=3$. If $G+H$ is 1-planar and $G$ contains a vertex of degree three, then $G+H$ is of class $\mathcal C_2$.
\end{lemma}

\begin{proof}
In this case the graph $G$ contains $K_{3,1}$ as a subgraph. Hence, $G+H$ contains $K_{3,3,1}$ as a subgraph. The crossing number of $K_{3,3,1}$ is 3, see \cite{asano}. Therefore, from Lemma \ref{crossing}  it follows that  $K_{3,3,1}$ does not have a $\mathcal C_1$-drawing. Consequently, if its supergraph $G+H$ is 1-planar, then it must be of class $\mathcal C_2$.
\end{proof}

\begin{lemma}
The graph $C_4+3P_1$ is of class $\mathcal C_2$.
\end{lemma}

\begin{proof}
The join $C_4+3P_1$ is 1-planar, see \cite{chm}. From Lemma \ref{c1c2} it follows that $C_4+3P_1$ cannot be of class $\mathcal C_0$. Assume that it is of class $\mathcal C_1$. Color the edges of $C_4$ with red and the other edges of $C_4+3P_1$ with black (the edges which join vertices of $C_4$ and $3P_1$). Any drawing of $C_4+3P_1$ has at least two crossings which are incident with only black edges, since the black edges induces $K_{4,3}$. Therefore, any $\mathcal C_1$-drawing of $C_4+3P_1$ has exactly two crossings (see Lemma \ref{crossing}). This means that in any $\mathcal C_1$-drawing of $C_4+3P_1$ no red edge is crossed. The red cycle divides the plane into two parts. If all vertices of $3P_1$ belong to the same part, then we remove one of them, after that we insert the removed vertex to the other part and we join it with the vertices of $C_4$. Clearly, we again obtain a $\mathcal C_1$-drawing of $C_4+3P_1$. So we can assume that the inner part of $C_4$ contains exactly two vertices of $3P_1$. Consequently, all crossings are inside the red  $C_4$, since the black edges which are outside the red $C_4$ are incident with a common vertex and no red edge is crossed. Therefore, if we remove the vertex which lies outside the red $C_4$ we obtain a $\mathcal C_1$-drawing of a graph on six vertices (with two crossings), a contradiction (see Lemma \ref{c0c2}).
\end{proof}

\begin{lemma}
The graph $C_3\cup P_1+3P_1$ is of class $\mathcal C_2$.
\end{lemma}

\begin{proof}
The join $C_3\cup P_1+3P_1$ is 1-planar, see \cite{chm}. From Lemma \ref{c1c2} it follows that $C_3\cup P_1+3P_1$ cannot be of class $\mathcal C_0$. Assume that it is of class $\mathcal C_1$. Then any $\mathcal C_1$-drawing $D$ of $C_3\cup P_1+3P_1$ has exactly two crossings. Therefore, the associated plane graph $D^\times$ has 9 vertices and 19 edges. Any plane triangulation on 9 vertices has 21 edges. This implies that $D^\times$ has either a face of size 5 or two faces of size 4. If $D^\times$ has a face $f$ of size 5, then on the boundary of $f$ there are at least 3 true vertices (since false vertices cannot be adjacent). We claim that we can add two diagonals $e_1$, $e_2$ to $f$ which join only true vertices. This is not possible if and only if at least one of these edges is already in $C_3\cup P_1+3P_1$. Assume that $e_1$ is in $C_3\cup P_1+3P_1$. If it is crossed by an other edge, then by shifting $e_1$ to the inner part of $f$ we can decrease the number of crossings to one, which is not possible. If $e_1$ is not crossed, then its endvertices form a 2-vertex-cut in $D^\times$. In \cite{fm} it was proved, that the associated plane graph of a 3-connected 1-planar graph is also 3-connected. Since $C_3\cup P_1+3P_1$ is 3-connected (it contains a 3-connected induced subgraph $K_{4,3}$), it cannot contain a 2-vertex-cut.

If $D$ contains two faces of size 4, then we can proceed similarly as above.

Consequently, we can add two edges to $D^\times$ which join only true vertices. If at least one of these two edges, say $e_1$, joins two vertices of $C_3 \cup P_1$, then we obtain a $\mathcal C_1$-drawing of $G+3P_1$, where $G$ is a graph $C_3 \cup P_1$ with the edge $e_1$. Since $G$ contains a vertex of degree 3, Lemma \ref{3vertex} implies that $G+3P_1$ does not belong to the class $\mathcal C_1$, a contradiction. Therefore, the two edges $e_1,e_2$ must join vertices of $3P_1$. In this case we obtain a $\mathcal C_1$-drawing of $C_3\cup P_1+P_3$, what is impossible, since its subgraph $C_3+P_3$ is of class $\mathcal C_2$, see Lemma \ref{c3p3}.
\end{proof}

\begin{lemma}
The graph $P_3 \cup P_1+C_3$ is of class $\mathcal C_2$.
\end{lemma}

\begin{proof}
It follows from Lemma \ref{c3p3}.
\end{proof}

\subsection{The last case: $|V(G)|\ge5$ and $|V(H)|\le2$}

Note that the graphs $nP_1+2P_1=K_{n,2}$ and $nP_1+P_1=K_{n,1}$ are planar, hence they belong to $\mathcal C_0$. Therefore, if the graph $H$ has at most two vertices, then there exist graphs $G$ with arbitrarily many vertices such that the join $G+H$ is 1-planar. Hence, it is not possible to describe the classes  $\mathcal C_0$, $\mathcal C_1$ and $\mathcal C_2$ without additional restrictions on $G$.


\subsubsection{The maximum degree of $G$}

Let $\Delta(G)$ denote the maximum degree of a graph $G$.
\begin{lemma}\label{huha}
If $G+2P_1$ or $G+P_2$ is of class $\mathcal C_0$, then $\Delta(G)\le3$. Moreover, this bound is tight.
\end{lemma}

\begin{proof}
If $G$ has a vertex of degree at least four, then it contains $K_{4,1}$ as a subgraph. Therefore, $K_{4,3}$ is a subgraph of $G+H$. The crossing number of $K_{4,3}$ is two, therefore $K_{4,3}$ and it supergraph $G+H$ cannot be of class $\mathcal C_0$, see Lemma \ref{hmco}.

Now we show that the bound is sharp. Let $C_k=v_1v_2\dots v_kv_1$ be a cycle on $k\ge6$ vertices. The plane drawing of this cycle divides the plane into two parts. Insert the edges $v_1v_3$ and $v_4v_6$ into different parts. We obtain a graph $G_k$ which has $k$ vertices, $k+2$ edges and maximum degree three, moreover, if we put the vertices of $2P_1$ into different faces of size $k-1$, then we can easily obtain a $\mathcal C_0$-drawing of $G_k + 2P_1$.

Let $G_k^-$ be the graph obtained from $G_k$ by removing the edge $v_3v_4$. Clearly, $\Delta(G_k^-)=3$ and the graph $G_k^- +P_2$ has a $\mathcal C_0$-drawing. 
\end{proof}

\begin{lemma}
If $G+2P_1$ or $G+P_2$ is of class $\mathcal C_1$, then $\Delta(G)\le4$. Moreover, this bound is tight.
\end{lemma}

\begin{proof}
If $G$ has a vertex of degree at least five, then it contains $K_{5,1}$ as a subgraph. Therefore, $G+H$ contains $K_{5,3}$ as a subgraph, moreover, $K_{5,3}$ is  of class $\mathcal C_2$ (see the proof of Lemma \ref{53}). Consequently, the supergraph $G+H$ of $K_{5,3}$ cannot be of class $\mathcal C_1$.  

In Figure \ref{delta4} is a graph $G$ of maximum degree four and a $\mathcal C_1$-drawing of $G+2P_1$, therefore the bound is sharp.
\end{proof}

\begin{figure}
\centerline{
\begin{tabular}{c}
\includegraphics{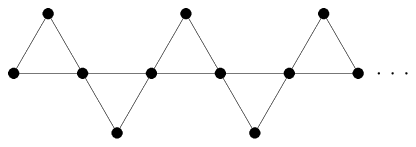}
\end{tabular}
$\longrightarrow$
\begin{tabular}{c}
\includegraphics{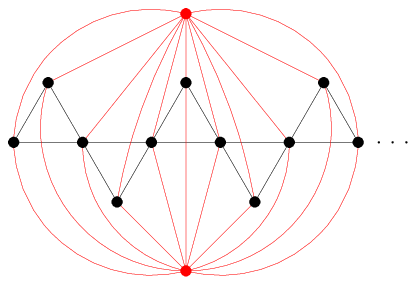}
\end{tabular}
}
\caption{The graph $G$ and a $\mathcal C_1$-drawing of $G+2P_1$.}
\label{delta4}
\end{figure}

\subsubsection{The number of edges of $G$}

\begin{theorem}\cite{zl}\label{edgesc0}
Let $W$ be a 1-planar graph of class $\mathcal C_0$. Then $|E(W)|\le 3,25|V(W)|-6$. Moreover this bound is tight.
\end{theorem}

%
%
%

\begin{lemma}\label{uff}
Let $W$ be an $n$-vertex 1-planar graph of class $\mathcal C_1$. Then every $\mathcal C_1$-drawing of $W$ has at most $0,6n -1,2$ crossings.
\end{lemma}

\begin{proof}
Let $D$ be a $\mathcal C_1$-drawing of $W$. Let $c$ denote the number of crossings in $D$. The associated plane graph $D^\times$ has $n+c$ vertices. Note that no two false vertices are adjacent in $D^\times$. Hence, we can extend $D^\times$ to a plane multi-triangulation $T$ by adding some edges into non-triangular faces of $D^\times$ which join only true vertices.

The obtained multi-triangulation $T$ has $2n+2c-4$ faces (Let $F(T)$ denote the face set of $T$. Clearly, $3|F(T)|=2|E(T)|$, since $T$ is a multi-triangulation. Combining this equality with Euler's formula $|V(T)|-|E(T)|+|F(T)|=2$, we obtain $|F(T)|=2|V(T)|-4$) and $4c$ of them are false.

Observe that every true edge in $T$ is incident with at most one false face. Therefore, the number of false faces cannot be greater than the number of true edges. On the other hand, every true edge is incident with a true face. Hence, the number of true edges is at most the treble of the number of true faces. Consequently, $4c\le 3t$, where $t$ denotes the number of true faces.

Therefore, $2n+2c-4=4c+t\ge 4c+\frac43 c$. Consequently, $2n-4\ge \frac{10}{3}c$, which implies $c\le \frac 35 n -\frac 65$.
\end{proof}

\begin{theorem}\label{edges}
Let $W$ be a 1-planar graph of class $\mathcal C_1$. Then $|E(W)|\le 3,6|V(W)|-7,2$. Moreover this bound is tight.
\end{theorem}

\begin{proof}
Let $D$ be a $\mathcal C_1$-drawing of $W$. Every crossing arises from two different edges. If we remove one crossed edge for each crossing in $D$, then we obtain a drawing without crossings, that is, a plane graph. Any plane graph on $n$ vertices has at most $3n-6$ edges. We removed $c\le 0,6n -1,2$ edges (see Lemma \ref{uff}), hence, the number of edges of $D$ (and also of $W$) is at most $3n-6+c\le 3,6n-7,2$.

Now we show that the bound is sharp.  We can construct graphs with the desired property  using the graphs depicted in Figure \ref{figure}. Let $S$ be a graph obtained from $G_1$ by inserting the graph $G_2$ into the central (green) triangle of $G_1$ (by identifying the green triangles).  Let $T$ be a graph obtained from $S$ by inserting the graph $G_1$ into the central (blue) triangle of $S$ (by identifying the blue triangles). This graph has 27 vertices and 90 edges, moreover, $27\cdot3,6-7,2=90$.
\end{proof}

\begin{figure}
\centerline{
\begin{tabular}{ccc}
\includegraphics{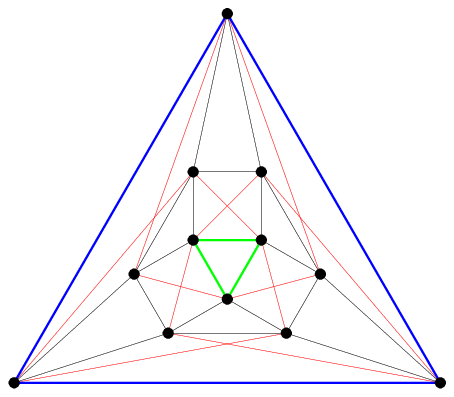}&&
\includegraphics{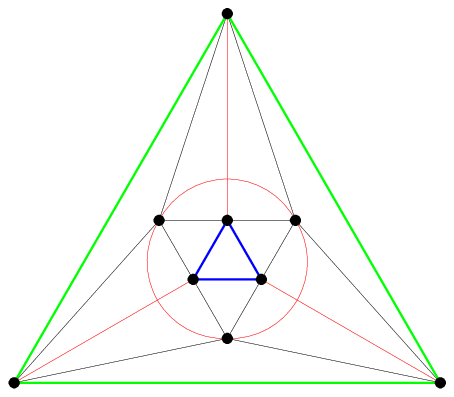}
\\$G_1$&&$G_2$
\end{tabular}
}
\caption{The graphs $G_1$ and $G_2$.}
\label{figure}
\end{figure}

\begin{lemma}\label{g2p1c0}
If $G+2P_1$ is of class $\mathcal C_0$, then $|E(G)|\le|V(G)|+2$. Moreover, this bound is tight.
\end{lemma}

\begin{proof}
Let $D$ be a $\mathcal C_0$-drawing of $G+2P_1$. Remove the two vertices of $2P_1$ from $D$. In such a way we obtain a drawing of $G$. First we show that this $\mathcal C_0$-drawing of $G$ contains no crossings. Assume that in this drawing of $G$ the edges $xy$ and $zw$ cross each other at $c$. Now consider a subgraph $\{xy,zw\}+2P_1$ of $G+2P_1$ in the drawing $D$. Lemma \ref{hmco} implies that this drawing of  $\{xy,zw\}+2P_1$ can contain at most one crossing.  Now we draw the edges $xz,zy,yw,wx$ to $\{xy,zw\}+2P_1$ such that they are crossing-free by following the edges that cross at $c$ from the endvertices until they meet in a close neighborhood of $c$. In this way we obtain a $\mathcal C_0$-drawing of $K_6$ minus one edge. Any planar graph on 6 vertices has at most 12 edges. The graph $K_6$ minus one edge has 6 vertices and 14 edges. Therefore, any drawing of $K_6$ minus one edge has at least two crossings, consequently, it cannot admit a $\mathcal C_0$-drawing (see Lemma \ref{hmco}), a contradiction.

Since the drawing $D$ without $2P_1$ is crossing-free, every crossed edge in $D$ has an endvertex in $2P_1$. Hence, $D$ contains at most two crossings (since it is a $\mathcal C_0$-drawing).  If we remove one crossed edge for each crossing in $D$, then we obtain a drawing without crossings. This implies
$|E(G)|+2|V(G)|=|E(G+2P_1)|\le 3|V(G+2P_1)|-6+2=3|V(G)|+2$,
which proves the claim.

To see that the bound is sharp it is sufficient to consider the graph $G_k$ defined in the proof of Lemma \ref{huha}.
\end{proof}

\begin{lemma}
If $G+P_2$ is of class $\mathcal C_0$, then $|E(G)|\le|V(G)|+1$. Moreover, this bound is tight.
\end{lemma}

\begin{proof}
We can proceed similarly as in the proof of Lemma \ref{g2p1c0}.
\end{proof}

\begin{lemma}
If $G+P_1$ is of class $\mathcal C_0$, then $|E(G)|\le 2,25|V(G)|-2,75$. Moreover, this bound is tight.
\end{lemma}

\begin{proof}
From Theorem \ref{edgesc0} we obtain
$|E(G)|+|V(G)|=|E(G+P_1)|\le 3,25|V(G+P_1)|-6=3,25|V(G)|-2,75$ 
which proves the claim.

Now we prove that the bound is sharp. Put $n = 2k$ with $k \geq 2$ being even, take two paths $a_1 a_2\dots a_{k-1}a_k, b_1 b_{2} \dots b_{k-1}$ and, for each $i \in \{1,\dots, k-1\}$, add new edges $a_ib_i, a_{i+1}b_i$ and the edge $a_{k-2}a_k$; in addition, for each even $j \in \{2,\dots, k-2\}$, add new edges $b_ja_{j-1}$. The resulting graph $G_{n-1}$ has $n-1$ vertices and $2,25(n-1)-2,75$ edges and a 1-planar drawing in which the edges $a_ib_{i+1},a_{i+1}b_i$ cross, for each odd $i\in\{1, \dots, k-3\}$ and the other edges are crossing-free (see Figure \ref{c0}). If we put a new vertex  $v$ into the outerface of $G_{n-1}^\times$ and join it with all vertices of $G_{n-1}^\times$ such that the edge $va_{k-1}$ cross the edge $b_{k-1}a_k$ and the other edges incident with $v$ are crossing-free, then we obtain a $\mathcal C_0$-drawing of $G_{n-1} + P_1$.
\end{proof}

\begin{figure}
\centerline{
\includegraphics{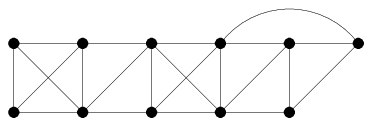}
}
\caption{The graph $G_{11}$.}
\label{c0}
\end{figure}

\begin{lemma}
If $G+2P_1$ is of class $\mathcal C_1$, then $|E(G)|\le 1,6 |V(G)|$.
\end{lemma}

\begin{proof}
From Theorem \ref{edges} we obtain
$|E(G)|+2|V(G)|=|E(G+2P_1)|\le 3,6|V(G+2P_1)|-7,2=3,6|V(G)|$
which proves the claim.
\end{proof}

\begin{lemma}
There is a graph $G$ with $|E(G)|=1,5|V(G)|$ such that $G+2P_1$ is of class $\mathcal C_1$.
\end{lemma}

\begin{proof}
Let $C=v_1v_2\dots v_{4\ell}v_1$ be a cycle on $4\ell\ge8$ vertices. The plane drawing of this cycle divides the plane into two parts. Add the edges $v_{4k-2}v_{4k}$, $k=1,\dots,\ell$, to the inner part and the edges $v_{4\ell}v_2$,$v_{4k}v_{4k+2}$, $k=1,\dots, \ell-1$, to the outer part. In such a way we obtain a graph $G$ with $4\ell$ vertices and $6\ell$ edges. Moreover, $G+2P_1$ has a $\mathcal C_1$-drawing.
\end{proof}

\begin{lemma}
If $G+P_1$ is of class $\mathcal C_1$, then $|E(G)|\le 2,6 |V(G)|-3,6$.
\end{lemma}

\begin{proof}
From Theorem \ref{edges} we obtain
$|E(G)|+|V(G)|=|E(G+P_1)|\le 3,6|V(G+P_1)|-7,2=3,6|V(G)|-3,6$
which proves the claim.
\end{proof}

\begin{lemma}
There is a graph $G$ with $|E(G)|=2,4|V(G)|-3,8$ such that $G+P_1$ is of class $\mathcal C_1$.  
\end{lemma}

\begin{proof}
Let $G_1$ be a graph depicted in Figure \ref{c1}. Let $G_k$, $k\ge 2$, be a graph obtained from $G_{k-1}$ and $G_1$ by identifying the edges $v_1v_2$ of $G_{k-1}$ and $u_1u_2$ of $G_1$. The graph $G_k$, $k\ge 2$, has  $3k+1$ vertices of degree three, $k$ vertices of degree six, $k-1$ vertices of degree nine and 2 vertices of degree four. Therefore, it has $12k+1$ edges. On the other hand, this graph has $5k+2$ vertices. Consequently, $|E(G_k)|=2,4|V(G_k)|-3,8$.

The graph $G_k+P_1$ has a $\mathcal C_1$-drawing, since $G_k$ is of class $\mathcal C_1$ and all true vertices of $G_k^\times$ are incident with the outer face.
\end{proof}

\begin{figure}
\centerline{
\begin{tabular}{ccc}
\includegraphics{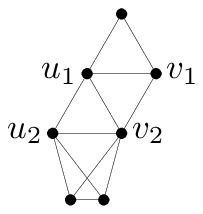}&&
\includegraphics{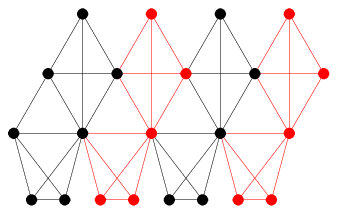}
\\$G_1$&&$G_4$
\end{tabular}
}
\caption{The graphs $G_1$ and $G_4$.}
\label{c1}
\end{figure}

\section{Conclusion}
In this paper we proved that every 1-planar graph is of class $\mathcal C_i$ for some $i\in\{0,1,2\}$. After that we proved that the join $G+H$ is of class $\mathcal C_0$ if and only if the pair $[G,H]$ is subgraph-majorized (that is, both $G$ and $H$ are subgraphs of graphs of the major pair) by one of pairs $[C_3, P_2 \cup P_1], [P_3,P_3]$ and is of class $\mathcal C_1$ if and only if the pair $[G,H]$ is subgraph-majorized by one of pairs $[2P_2\cup C_3], [P_4,P_3]$ in the case when both factors of the graph join have at least three vertices.

In \cite{chm} it was proved that the join $G+H$ is 1-planar if and only if the pair $[G,H]$ is subgraph-majorized by one of pairs $[C_3 \cup C_3,C_3], [C_4,C_4], [C_4,C_3], [K_{2,1,1},P_3]$. Therefore we have full characterization of 1-planar joins in the case when both factors have at least three vertices.

Finally, we proved several necessary conditions for the bigger factor in the case when the smaller one has at most two vertices; two of them improve the results of Zhang \cite{z}.

\end{document}